\documentclass[11pt]{amsart}

\usepackage{latexsym}
\usepackage{amssymb,amsfonts,amsmath,amsthm}
\usepackage{graphicx}

\newcommand{\eps}{\varepsilon}

\newcommand{\ex}{\mathbb{E}}

\newcommand{\cH}{\mathcal{H}}

\newcommand{\cA}{\mathcal{A}}
\newcommand{\cR}{\mathcal{R}}

\newcommand{\A}{\mathcal{A}}

\newcommand{\cG}{\mathcal{G}}

\newtheorem{theorem}{Theorem}  

\newtheorem{lemma}[theorem]{Lemma}

\newtheorem{corollary}[theorem]{Corollary}

\newtheorem{proposition}[theorem]{Proposition}

\def\COMMENT#1{}
\def\TASK#1{}

\numberwithin{theorem}{section}
\numberwithin{equation}{section}

\newdimen\margin   
\def\textno#1&#2\par{%
   \margin=\hsize
   \advance\margin by -4\parindent
          \setbox1=\hbox{\sl#1}%
   \ifdim\wd1 < \margin
      $$\box1\eqno#2$$%
   \else
      \bigbreak
      \hbox to \hsize{\indent$\vcenter{\advance\hsize by -3\parindent
      \it\noindent#1}\hfil#2$}%
      \bigbreak
   \fi}
\def\noproof{{\unskip\nobreak\hfill\penalty50\hskip2em\hbox{}\nobreak\hfill%
       $\square$\parfillskip=0pt\finalhyphendemerits=0\par}\goodbreak}
\def\endproof{\noproof\bigskip}
\def\proof{\removelastskip\penalty55\medskip\noindent{\bf Proof. }}

\title{On P\'osa's conjecture for random graphs}

\author{Daniela K\"uhn and Deryk Osthus}
\thanks {D.~K\"uhn was supported by the ERC, grant no.~258345.}

\begin{document}

\begin{abstract}
The famous P\'osa conjecture states that every graph of minimum degree at least $2n/3$ contains the square of a Hamilton cycle.
This has been proved for large $n$ by Koml\'os, Sark\"ozy and Szemer\'edi.
Here we prove that if $p \ge n^{-1/2+\eps}$, then asymptotically almost surely, the binomial random graph $G_{n,p}$
contains the square of a Hamilton cycle. 
This provides an `approximate threshold' for the property in the sense that the result fails to hold if $p\le n^{-1/2}$. 
\end{abstract}

\date{\today}

\maketitle 

\section{Introduction} 
The $k$th power of a cycle $C$ is obtained by including an edge between all pairs of vertices whose distance on $C$ is at most $k$.
The P\'osa-Seymour conjecture states that every graph $G$ on $n$ vertices with minimum degree at least $kn/(k+1)$ contains the $k$th power of a Hamilton cycle.
(Here the case $k=2$ was conjectured by P\'osa and the general case was later conjectured by Seymour.)
This beautiful conjecture was proved for large $n$ by Koml\'os, Sark\"ozy and Szemer\'edi~\cite{KSSz98}.
The case $k=1$ of course corresponds to Dirac's theorem~\cite{Dirac} on Hamilton cycles.
For $k=2$, there have been significant improvements in the bound on $n$ that is required (see e.g.~\cite{CDH}).
More generally, many other recent advances have been made on embedding spanning subgraphs in dense graphs (see e.g.~\cite{BCCsurvey} for a survey).
For instance, recall that $G$ has an \emph{$F$-factor} if $G$ contains $\lfloor |G|/|F|\rfloor$ vertex-disjoint copies of $F$.
The famous Hajnal-Szemer\'edi theorem~\cite{HSz} states that every graph with minimum degree at least $kn/(k+1)$ contains a $K_{k+1}$-factor.
More generally, K\"uhn and Osthus~\cite{KOmatch} determined  the minimum degree that $G$ needs to have to ensure the existence of an $F$-factor in $G$
(up to an additive constant depending on~$F$).

It is natural to ask for probabilistic analogues of these results, i.e.~given a graph $H$ on $n$ vertices, how large does $p$ have to be to ensure that 
$G_{n,p}$ a.a.s.~contains a copy of $H$?
Here $G_{n,p}$ denotes the binomial random graph on $n$ vertices with edge probability $p$ and
we say that a property $A$ holds a.a.s.~(asymptotically almost surely), if the probability that $A$ holds tends to $1$ as $n$ tends to infinity.
(Note that formally one actually needs to ask the above question for a sequence of graphs $H_i$ whose order tends to infinity.)

This turns out to be a surprisingly difficult problem, and the answer is known for very few (families of) graphs $H$.
A notable exception is the seminal result of Johansson, Kahn and Vu~\cite{JKV08}, who determined the `approximate' threshold for the existence of an $F$-factor. 
So this is a probabilistic version of the result in~\cite{KOmatch} mentioned above.
Also, Riordan~\cite{riordan} obtained a very general result, which gives a bound that can be applied to every graph $H$.
As a corollary, he obtained the threshold for the existence of a spanning hypercube in $G_{n,p}$
and several kinds of spanning lattices, e.g.~the square grid.
His result can be applied to powers of Hamilton cycles to give the following result (see Section~\ref{sec:riordan} for the straightforward details):

\begin{theorem} \label{thm:riordan}
Let  $k\geq 2$ be fixed. Suppose that $p n^{1/k} \to \infty$ and $p n^{1/3} \to \infty$.
Then a.a.s.~$G_{n,p}$ contains the $k$th power  of a Hamilton cycle.
\end{theorem}
A simple first moment argument shows that this result gives the correct threshold for $k \ge 3$. Indeed, note that the number of edges in 
the $k$th power of a cycle of length $n > 2k$ is $kn$. So if $n > 2k$ and $p \le n^{-1/k}$, it follows that 
the expected number of appearances of the $k$th power of a Hamilton cycle in $G_{n,p}$ is at most
$n!p^{kn} \leq  ( n p^{k}/2)^n = o(1)$. 

However, for squares (i.e.~when $k=2$) Theorem~\ref{thm:riordan} does not give the correct answer.
Indeed, the above first moment argument suggests that the threshold should be close to $n^{-1/2}$.
Our main result is an `approximate' threshold, i.e.~our bound on $p$ is tight up to a factor of $n^{\eps}$, where $\eps>0$ is arbitrary.
Our argument works for higher powers in the same way as it does for squares, so we formulate our proof for arbitrary $k \ge 2$.

\begin{theorem} \label{thm:main}
Let $\eps>0$ and $k\geq 2$ be fixed. Suppose that $p=p(n)\geq n^{-1/k + \eps}$.
Then a.a.s.~$G_{n,p}$ contains the $k$th power $C^k$ of a Hamilton cycle.
\end{theorem}
Note that Theorems~\ref{thm:riordan} and~\ref{thm:main} as well as the result on $F$-factors in~\cite{JKV08} (see Theorem~\ref{thm:JKV}) imply that the threshold  
for a $K_{k+1}$-factor is much smaller than that for the $k$th power of a Hamilton cycle.
So this is different from the `deterministic' setting described earlier, where the minimum degree conditions are the same.

We now discuss some further related results on embedding spanning subgraphs in $G_{n,p}$.
The case of Hamilton cycles (i.e.~when $k=1$) has been studied successfully and in great detail.
In particular, a classical result of Koml\'os and Szemer\'edi~\cite{KomSze} and Korshunov~\cite{Korsh} 
implies that the threshold function for the existence of a Hamilton cycle is 
$\log n/n$. In fact, much more is true: a celebrated result of 
Bollob\'as~\cite{BolHam} and~Ajtai, Koml\'os and Szemer\'edi~\cite{AKSHam} states that the hitting time for the emergence of a Hamilton 
cycle on $n$ vertices coincides a.a.s.~with the hitting time of the property that the minimum degree is at least~2.
(An algorithmic version of this result was later proved by Bollob\'as, Fenner and Frieze~\cite{BFF}.) 
On the other hand, the expected number of Hamilton cycles already tends to infinity when $np \rightarrow \infty$. 
So the existence of vertices of degree less than two in $G_{n,p}$ can be viewed as a `local obstruction' to the existence of 
a Hamilton cycle in $G_{n,p}$. For $k \ge 3$, Theorem~\ref{thm:riordan} shows that there are no `local obstructions'.
It seems natural to conjecture that the case of squares is similar, i.e. that the threshold for the square of a Hamilton cycle in $G_{n,p}$
is at $p=n^{-1/2}$.

Another class of  subgraphs which has received much attention is that of spanning trees.
The best general result is due to Krivelevich~\cite{krivtrees}, who showed (amongst other results)
that if $T$ is any tree on $n$ vertices of bounded maximum degree and $p \ge n^{-1+\eps}$, then a.a.s.~$G_{n,p}$ contains a copy of $T$.
It seems likely that the term $n^\eps$ in this result can be replaced by a much smaller function.
This is supported by several results on certain classes of trees. For instance, the threshold for a Hamilton path is $p=\log n/n$ by the above results on Hamilton cycles.
As another example, Hefetz, Krivelevich, and Szab\'o~\cite{HKS} showed that $p=\log n/n$ is the (sharp)
threshold for a tree $T$ having a linear number of leaves.

In the probabilistic setting, it is also natural to ask for `universality' results. Again, this is a question where much progress has been made recently.
Given a graph $G$ and a family of graphs $\cH$, we say that a graph $G$ is \emph{$\cH$-universal} if $G$ contains every member of $\cH$ as a subgraph.
An important case is  when $\cH=\cH(n,\Delta)$ consists of all graphs on $n$ vertices with maximum degree at most $\Delta$.
Here the best bound is due to Dellamonica, Kohayakawa, R\"odl and Ruci\'nski~\cite{DKRR}, who showed that if $p \gg n^{-1/2\Delta} \log ^{1/\Delta}n$, 
then a.a.s.~$G_{n,p}$ is $\cH(n,\Delta)$-universal. Note that the $k$th power of a Hamilton cycle on $n>2k$ vertices has maximum degree $2k$.
So the bounds one obtains for this case are significantly weaker than the ones given by Theorems~\ref{thm:riordan} and~\ref{thm:main}.

The proof in~\cite{riordan} is based on the second moment method.
Instead, our proof is based on the `absorbing method', which was introduced as a general method by R\"odl, Ruci\'nski and Szemer\'edi~\cite{RRS}
(the underlying idea was also used earlier,  e.g.~by Krivelevich~\cite{Krivtriangles}).
The method has proved to be an extremely versatile tool
for embedding various types of spanning subgraphs in dense graphs. 
Though additional difficulties arise in the context of (sparse) random graphs, we believe that the method has significant further potential
in this setting.

This paper is organized as follows. After introducing some notation, we define an `absorber', which will be the crucial concept
for extending the $k$th power of an almost spanning cycle into the $k$th power of a 
Hamilton cycle. We then describe the proof of Theorem~\ref{thm:main} in Section~\ref{sec:mainproof}, under the assumption that
Lemmas~\ref{lem:Absorber_Factor},~\ref{lem:linking} and~\ref{lem:pathfactor} hold. 
Section~\ref{sec:mainproof} also contains an informal overview of the proof.
These lemmas are proved in the subsequent sections.
In the short final section, we show how Theorem~\ref{thm:riordan} follows from the more general result in~\cite{riordan}.


\section{Notation}
We write $|G|$ and sometimes also $v_G$ for the number of vertices of a graph $G$. We write $e(G)$ and sometimes also $e_G$ for the
number of edges of $G$. We say that two graphs $H$ and $G$ are disjoint if they are vertex-disjoint. 
Given graphs $G$ and $H$, an $H$-factor in $G$ is a collection of $\lfloor |G|/|H|\rfloor$ pairwise disjoint copies
of $H$ in~$G$.

We denote the path on $n$ vertices by $P_n$.
The \emph{distance} between two vertices $x$ and $y$ in a graph $G$ is the length (i.e.~the number of edges) of the shortest path between $x$ and $y$.
The \emph{$k$th power} of a graph $G$ is the graph $G^k$ whose vertex set is $V(G)$ and in which two vertices $x,y\in V(G)$
are joined by an edge if and only if the distance between $x$ and $y$ in $G$ is at most $k$.
So $P^k_n$ denotes the $k$th power of~$P_n$. Suppose that $n\ge 2k$ and that $P_n=x_1\dots x_n$. We will view $x_1$ as the first vertex of $P_n$
and $x_n$ as its final vertex. The \emph{initial endsequence} of $P^k_n$ is the sequence $x_1,\dots,x_k$ and the
\emph{final endsequence} of $P^k_n$ is the sequence $x_{n-k+1},\dots,x_n$.

Suppose that $A=(a_1,\dots,a_k)$ and $B=(b_1,\dots,b_k)$ are two sequences of vertices such that all these $2k$ vertices are
distinct from each other. An \emph{$(A,B)$-linkage} $R$  is defined as follows: 
let $R'$ be the $k$th power of a path such that the initial endsequence of $R'$ is $A$ and the final endsequence of $R'$ is $B$. 
Then we obtain $R$ by removing all edges within $A$ and within $B$.
We will use the notion of linkages to join
up $k$th powers of paths into longer ones. More precisely, suppose that $Q$ and $Q'$ are $k$th powers of paths
which are pairwise disjoint, that $A$ is the final endsequence of~$Q$, that $B$ is the initial endsequence of~$Q'$ and
that $R$ is an $(A,B)$-linkage which meets  $V(Q)\cup V(Q')$
only in $A\cup B$. Then $Q\cup R\cup Q'$ is again the $k$th power of a path.

We will omit floors and ceilings whenever this does not affect the argument. We write $\log n$ for the natural logarithm and
$\log ^a n:=(\log n)^a$.


\section{Absorbers}\label{sec:absorber}

The aim of this section is to define an \emph{absorber}, which is the main tool in our proof of Theorem~\ref{thm:main}.
Roughly speaking, an absorber $A$ will be the union of the $k$th power $P^k$ of a path $P$ and the $k$th power $(P')^k$ of a path $P'$
such that the following two properties are satisfied:
\begin{itemize}
\item The two endsequences of $P^k$ are the same as the two endsequences of $(P')^k$.
\item $V(P')$ is obtained from $V(P)$ by adding one extra vertex $v$ (which we call the absorbtion vertex).
\end{itemize}
Thus if we can find the $k$th power $C^*$ of some cycle which contains $P^k$ as a subgraph but does not contain~$v$, then
we can `absorb' $v$ into $C^*$ by replacing $P^k$ with $(P')^k$. When defining the absorber, we have to make sure that our random graph
$G_{n,p}$ a.a.s.~contains many disjoint copies of this absorber. A result of Johansson, Kahn and Vu (Theorem~\ref{thm:JKV} below) implies
that the latter will be the case if the $1$-densities of all subgraphs of the absorber are not too large.
(This will turn out to be true if the parameters $j$ and $\ell$ below satisfy $k\ll j\ll \ell$.)

More precisely, for all $k\ge 2$,  $j\geq 3$ and $\ell \geq 2k$, we will now define the $(j,\ell,k)$-\emph{absorber} $A_{j,\ell,k}$. 
Consider first a path $P$ on $s$ vertices, where $s:=j(2\ell+4)+\ell$, and a vertex~$v$ that does not belong to~$P$. 
We call $P$ the \emph{spine} of the absorber and~$v$ its \emph{absorbtion vertex}. We will view one endvertex of $P$ as its first vertex
and the other endvertex of $P$ as its last vertex. This induces an order of the vertices on~$P$.
Split $P$ into $j+1$ consecutive disjoint segments $S_1,\dots, S_{j+1}$ such that $S_i$ has $2\ell + 4$ vertices for each $i=1,\dots,j$
and $S_{j+1}$ consists of the final $\ell$ vertices of~$P$. 
For $i=1,\dots,j$, in $S_i$ we label the $(\ell +1)$st, the $(\ell+2)$nd, the $(2\ell + 3)$rd and the $(2\ell+4)$th vertices by 
$a_{i,1}$, $a_{i,2}$, $b_{i,1}$ and $b_{i,2}$, respectively. We call these special vertices \emph{junctions}.

We add the edges $a_{1,1}v$ and $vb_{1,2}$. 
For every $i=1,\dots,j-2$, we add the edges $a_{i,2}b_{i+1,2}$ and $b_{i,1}a_{i+1,1}$. 
Finally,  we add the edges $a_{j,2}b_{j,2}$, $a_{j-1,2}a_{j,1}$ and $b_{j-1,1}b_{j,1}$.  
We will be referring to the resulting graph (consisting of the spine~$P$, the absorbtion vertex $v$ and the edges incident to the junctions and to $v$
which we added) as the \emph{skeleton} of the absorber.

It is not hard to see that the graph $P'$ obtained from the skeleton by deleting the edges $a_{i,1}a_{i,2}$ and $b_{i,1}b_{i,2}$ for all $i=1,\dots,j$
is a path with $V(P')=V(P) \cup \{v\}$ and with the same endvertices as the spine~$P$ (see Figure~\ref{fig:absorber} for the case when $j=4$).%
    \COMMENT{More formally, $P'= [1 a_{1,1}]v \bigcup \left( \bigcup_{i=1}^{j-1} [b_{i,2} a_{i+1,1}] [b_{i,1} a_{i,2}] \right) \bigcup 
[b_{j-1,2}a_{j,1}] [a_{j-1,2} b_{j-1,1}] [b_{j,1} a_{j,2}] [b_{j,2}s]$.}
\begin{figure}[htp] 
\begin{center}
\includegraphics[scale=0.65]{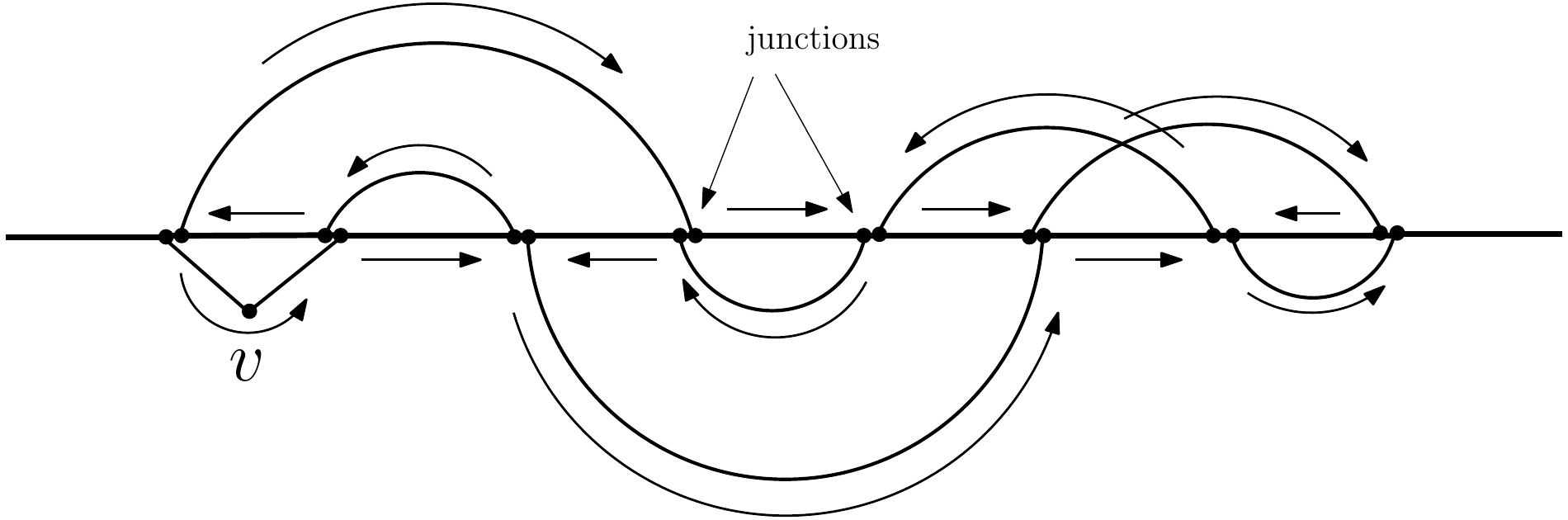}
\caption{The skeleton of a $(4,\ell)$-absorber. The path $P'$ is indicated by the arrows.}
\label{fig:absorber}
\end{center}
\end{figure}
\begin{figure}[htp] 
\begin{center}
\includegraphics[scale=0.55]{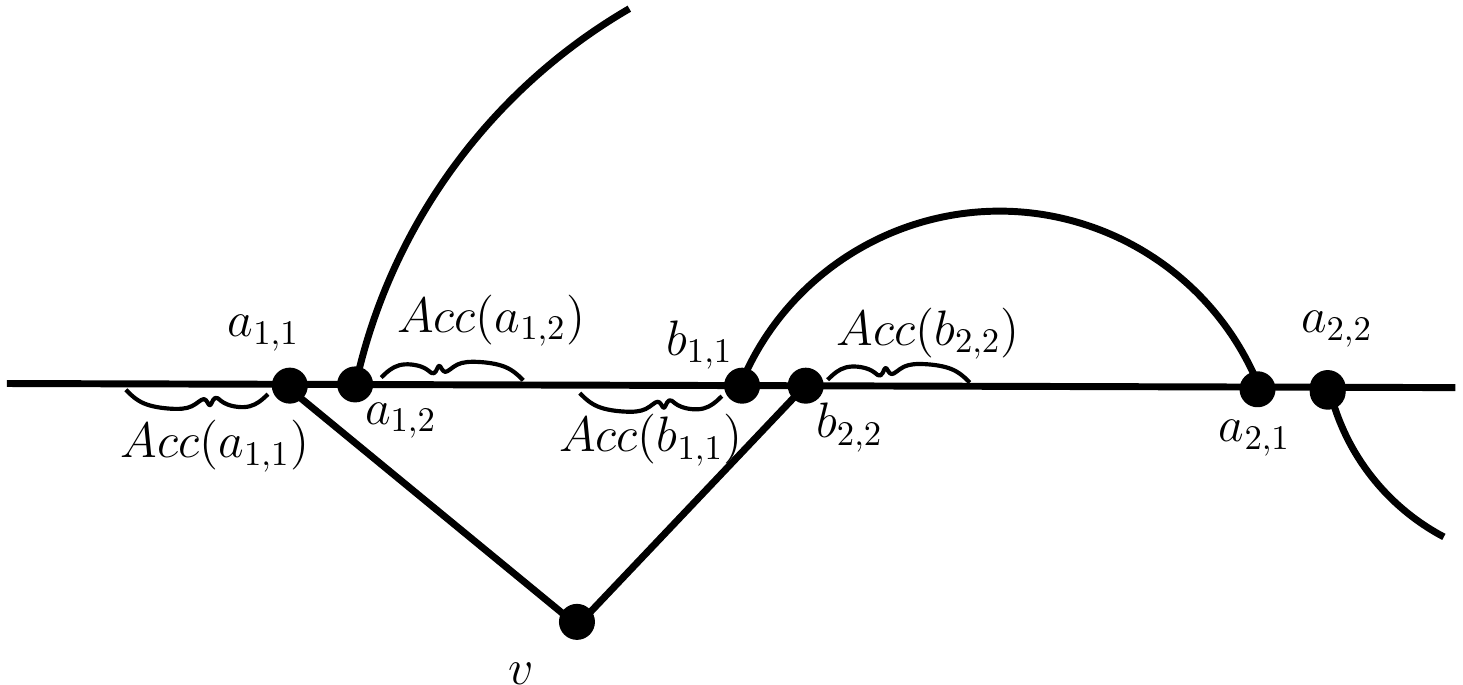}
\caption{Junctions and access vertices of a $(4,\ell)$-absorber.}
\label{fig:absorberdetail}
\end{center}
\end{figure}
We call $P'$ the \emph{augmented path} of the absorber and the edges in $E(P')\setminus (E(P)\cup\{a_{1,1}v,vb_{1,2}\})$ the \emph{junction edges}.
We define the $(j,\ell,k)$\emph{-absorber} $A_{j,\ell,k}$ to be%
    \COMMENT{Note that this is different from $(P\cup P')^k$.}
$P^k \cup (P')^k$. 
The \emph{first endsequence of} $A_{j,\ell,k}$ is the first endsequence of $P^k$ (and thus of $(P')^k$) and the \emph{final endsequence of} $A_{j,\ell,k}$ is
the final endsequence of $P^k$ (and thus of $(P')^k$). 

Given a junction $a$, let ${\rm Acc}(a)$ be the set consisting of $a$ as well as all the $k-1$ vertices that have distance at most
$k-1$ from $a$ in \emph{both} $P$ and $P'$ (see also Figure~\ref{fig:absorberdetail},
where these sets are marked for four of the junctions). Call the vertices in ${\rm Acc}(a)$ \emph{access vertices} associated with~$a$. 
Note that the following properties hold:
\begin{itemize}
\item[(A1)] Let $ab$ be any junction edge, where $a$ is the predecessor of $b$ on~$P'$. Then the subpath $Q_a$ of $P'$ induced by $a$ and the $\ell$ vertices preceding $a$ on~$P'$
is also a subpath of $P$ and ${\rm Acc}(a)$ is the set of all those vertices having distance at most $k-1$ from~$a$ on $Q_a$.
Similarly, the subpath $Q_b$ of $P'$ induced by $b$ and the $\ell$ vertices succeeding $b$ on~$P'$
is also a subpath of $P$ and ${\rm Acc}(b)$ is the set of all those vertices having distance at most $k-1$ from~$b$ on $Q_b$.
\item[(A2)] $a_{1,1}vb_{1,2}$ is a subpath of $P'$. The subpath $Q_{a_{1,1}}$ of $P'$ induced by $a_{1,1}$ and the $\ell$ vertices preceding $a_{1,1}$ on~$P'$
is also a subpath of $P$ and ${\rm Acc}(a_{1,1})$ is the set of all those vertices having distance at most $k-1$ from~$a_{1,1}$ on $Q_{a_{1,1}}$.
Similarly, the subpath $Q_{b_{1,2}}$ of $P'$ induced by $b_{1,2}$ and the $\ell$ vertices succeeding $b_{1,2}$ on~$P'$
is also a subpath of $P$ and ${\rm Acc}(b_{1,2})$ is the set of all those vertices having distance at most $k-1$ from~$b_{1,2}$ on $Q_{b_{1,2}}$.
\item[(A3)] The graph consisting of all junction edges, of the path $a_{1,1}vb_{1,2}$ and of all the edges $a_{i,1}a_{i,2}$, $b_{i,1}b_{i,2}$
(for all $i=1,\dots,j$) is a cycle.
\end{itemize}
(A1) and~(A2) together with the fact that $\ell\ge 2k$ imply that every edge $e\in E(A_{j,\ell,k})\setminus E(P^k)$ satisfies precisely one of the following conditions:
\begin{itemize}
\item There is precisely one junction edge $ab$ such that $e$ joins some vertex in ${\rm Acc}(a)$
to some vertex in ${\rm Acc}(b)$.
\item $e$ joins some vertex in ${\rm Acc}(a_{1,1})\cup \{v\}$ to some vertex in ${\rm Acc}(b_{1,2})\cup \{v\}$.
\end{itemize}
In the first case we say that $e$ is \emph{associated with~$ab$} (so $ab$ itself is associated with $ab$)
and in the second case we say that $e$ is \emph{associated with}~$v$.
Note that for every junction edge $ab$ there are precisely $\binom{k+1}{2}$ edges associated with~$ab$.
Indeed, let $a_k:=a$ and for each $i=1,\dots,k-1$ let $a_i$ be the vertex of distance $i$ from $a$ on $Q_a$,
where $Q_a$ is as defined in~(A1). (So  ${\rm Acc}(a)=\{a_1,\dots,a_k\}$.) Then $a_i$ has precisely $i$ neighbours in ${\rm Acc}(b)$.
Similarly, precisely $\binom{k+1}{2}+k$ edges are associated%
    \COMMENT{We get the extra $+k$ since $v$ is a neighbour of each vertex in $Acc(a_{1,1})$ and each vertex in $Acc(b_{1,2})$ and so we don't count these edges twice}
with $v$. Since there are $2j-1$ junction edges, altogether this shows that
\begin{equation}\label{eq:edgesabs}
e(A_{j,\ell,k})= e(P^k)+2j\binom{k+1}{2}+k.
\end{equation}


\section{Proof of Theorem~\ref{thm:main}} \label{sec:mainproof}

Since the property of containing the $k$th power of a Hamilton cycle is monotone it suffices to show that a.a.s.~$G_{n,p^*}$ contains
the $k$th power of a Hamilton cycle, where
$$p^*=p^*(n):= n^{-1/k + \eps_*}.$$
Here $\eps_*$ is fixed and we assume that
\begin{equation}\label{eq:eps*}
\eps_*\le \frac{1}{10^4 k}.
\end{equation}
So in particular $p^*=o(1)$. We shall consider a multiple round exposure of $G_{n,p^*}$. More precisely,  
we will expose $G_{n,p^*}$ in four rounds considering four independent random graphs $G_{n,p^*_1},\dots, G_{n,p^*_4}$, where $p^*_1=\dots=p^*_4$.
Thus $p^*_i =(1+o(1))p^*/4\ge n^{-1/k + \eps_*/2}$ for all $i=1,\dots,4$. 

Roughly speaking, the strategy of our proof is as follows. We will first use $G_{n,p^*_1}$ to find a collection $\mathcal{A}$ of pairwise disjoint
absorbers which cover about $n/3$  vertices. Let $A$ denote the set consisting of all absorbtion vertices of all these absorbers.
We use $G_{n,p^*_2}$ to connect the $k$th powers of the spines of the absorbers in $\A$ into the $k$th power $Q_\A$ of a path.
To do this we will only use vertices which are not covered by the absorbers in~$\A$.
Moreover, $V(Q_\A)\cup A$ will contain at most $2n/3$ vertices.
Let $S':=[n]\setminus (V(Q_\A)\cup A)$ denote the set of uncovered vertices.
We will use $G_{n,p^*_3}$ to cover $S'$ by a collection $\mathcal{P}$ consisting of not too many $k$th powers of pairwise disjoint paths.
Finally, we will use $G_{n,p^*_4}$ connect all the paths in $\mathcal{P}$ as well as $Q_{\A}$ into the $k$th power $C^*$ of a cycle.
To do this we will only use vertices in $A$. Let $A''\subseteq A$ be the vertices not used in this step. Since $A''$ consists of
absorbtion vertices, we can `absorb' all the vertices of $A''$ into $C^*$ to obtain the $k$th power of a Hamilton cycle.
More precisely, for each $v\in A''$ let
$A_v$ denote the unique absorber in $\A$ that contains~$v$. Then the subgraph obtained from $C^*$ by replacing the $k$th power of the spine
of $A_v$ with the $k$th power of the augmenting path of $A_v$ (for each $v\in A''$) is the $k$th power of a Hamilton cycle in~$G_{n,p^*}$.

After outlining our strategy, let us now return to the actual proof. We will use the next lemma (which is proved in Section~\ref{sec:Absorber_Factor})
in order to find the collection $\A$ of absorbers in $G_{n,p^*_1}$. 

\begin{lemma} \label{lem:Absorber_Factor} 
For each $\eps >0$ and each integer $k\geq 2$, there exist integers $j\ge 3$ and $\ell_0 \geq 2k$ such that
whenever $\ell\ge \ell_0$ and $p=p(n)\ge n^{-1/k + \eps}$, then a.a.s. $G_{n,p}$ contains an $A_{j,\ell,k}$-factor.
\end{lemma}
Let $j=j(k,\eps_*/2)$ and $\ell_0=\ell_0(k,\eps_*/2)$ be as in Lemma~\ref{lem:Absorber_Factor}.
Set
\begin{equation}\label{eq:ell}
\ell:=\max\{ \ell_0, \lceil 1/\eps_*^2\rceil\}.
\end{equation}
Then Lemma~\ref{lem:Absorber_Factor} implies that a.a.s.~$G_{n,p^*_1}$ contains an $A_{j,\ell,k}$-factor.
So we may assume that such a factor exists.
Let $s:=j(2\ell+4)+\ell$ and note that $|A_{j,\ell,k}|=s+1$.
Let $\A$ be a collection of $n/(3(s+1))$ copies of $A_{j,\ell,k}$ in this $A_{j,\ell,k}$-factor and let $A$ denote the set of absorbtion vertices
in all these copies. 
(So the assertion of Lemma~\ref{lem:Absorber_Factor} is far stronger than we need it to be -- see the discussion after Theorem~\ref{thm:JKV}.)
Note that the absorbers in $\A$ cover $n/3$ vertices of $G_{n,p^*}$.
Let $S$ be a set of $n/3$ vertices not covered by these absorbers. As indicated above, our next aim is to use $G_{n,p^*_2}$
in order connect the absorbers in~$\A$, using some of the vertices in $S$. To do this, we will use the following lemma (which we prove in
Section~\ref{sec:linking}).

\begin{lemma}\label{lem:linking}
Suppose that $k\ge 2$, that $0<\eps<1/k$, that $p=p(n) \ge n^{-1/k+\eps}$ with $p(n)=o(1)$ and that $f \le \eps n/(60 k)$.
For each $i=1,\dots,f$ let $A_i$ and $B_i$ be sequences, each consisting of $k$ vertices in $[n]$, such that these $2f$ sequences
are pairwise disjoint. Then a.a.s.~$G_{n,p}$ contains pairwise disjoint $(A_i,B_i)$-linkages $R_i$ with $|R_i| \le \lceil 30k/\eps \rceil$
(for all $i=1,\dots,f$).
\end{lemma}
Choose an order of the absorbers in~$\A$. For each $i=1,\dots,|\A|-1$ let $A_i$ denote the final endsequence of the $i$th absorber in~$\A$
and let $B_i$ be the initial endsequence of the $(i+1)$st absorber in~$\A$.  
Let $S^*$ denote the union of $S$ together with all the vertices contained in one of these endsequences $A_i$ or $B_i$.
Note that $$
|\A|=\frac{n}{3(s+1)}= \frac{|S|}{s+1}\le \frac{|S|}{\ell}\stackrel{(\ref{eq:ell})}{\le} \eps_*^2 |S|\stackrel{(\ref{eq:eps*})}{\le} \frac{\eps_* |S|}{180k} \le  \frac{\eps_* |S^*|}{180k}
$$
and $p^*_2\ge n^{-1/k+\eps_*/2}\ge |S^*|^{-1/k+\eps_*/3}$. So
we may apply Lemma~\ref{lem:linking} (with $\eps_*/3$ playing the role of $\eps$) to see that a.a.s.~the random subgraph of $G_{n,p^*_2}$ induced by $S^*$
contains pairwise disjoint $(A_i,B_i)$-linkages~$R_i$ with $|R_i| \le \lceil 90k/\eps_* \rceil$ for all $i=1,\dots,|\A|-1$. 
So we may assume that such linkages exist. Let $Q_{\A}$ be the union of $R_1,\dots,R_{|\A|-1}$ and of the $k$th powers of the spines of all absorbers in $\A$.
Then $Q_{\A}$ is the $k$th power of a path whose initial endsequence is the initial endsequence of the first absorber in $\A$ and whose
final endsequence is the final endsequence of the last absorber in $\A$. Moreover, $Q_{\A}$ avoids the set $A$ of absorbtion vertices.

Let $S':=[n]\setminus (V(Q_\A)\cup A)$ be the set of uncovered vertices. Thus $|S'|\ge n/3$. Our next aim is to cover $S'$ with
not too many $k$th powers of paths. 
To simplify this step, first let $t:=|S'| \mod s^2$. 
Now remove $s^2-t$ vertices from $A$ and call the resulting set $A'$. 
Add these $s^2-t$ vertices to $S'$ and call the resulting set $S''$. So $|S''|$ is divisible by $s^2$.

The next lemma (which will be proved in Section~\ref{sec:pathfactor}) implies
that a.a.s.~the random subgraph of $G_{n,p^*_3}$ induced by $S''$ contains a $P^k_{s^2}$-factor~$\mathcal{P}$.
So we may assume that such a factor exists.

\begin{lemma}\label{lem:pathfactor} 
Suppose that $\eps>0$, that $k,r\ge 2$ and that $p=p(n) \ge n^{-1/k+\eps}$. Then a.a.s.~$G_{n,p}$ has a $P_r^k$-factor. 
\end{lemma}

Since $|S''|$ is divisible by $s^2$, all the vertices in $S''$ are covered by~$\mathcal{P}$.
We will now use $G_{n,p^*_4}$ to connect all the copies of $P^k_{s^2}$ in~$\mathcal{P}$ as well as $Q_\A$ into the $k$th power of
a cycle, using some of the vertices in~$A'$. To do this, we
choose an order of the copies of $P^k_{s^2}$ in~$\mathcal{P}$. For each $i=1,\dots,|\mathcal{P}|-1$ let $A'_i$ denote the final endsequence of the $i$th
copy of $P^k_{s^2}$ in~$\mathcal{P}$ and let $B'_i$ be the initial endsequence of the $(i+1)$st copy.
Let $A'_{|\mathcal{P}|}$ denote the final  endsequence of the last copy of $P^k_{s^2}$ in~$\mathcal{P}$ and let $B'_{|\mathcal{P}|}$ denote the initial endsequence
of~$Q_{\A}$. Finally, let $A'_{|\mathcal{P}|+1}$ denote the final endsequence of $Q_{\A}$ and let $B'_{|\mathcal{P}|+1}$ denote the initial endsequence
of the first copy of $P^k_{s^2}$ in~$\mathcal{P}$.
Let $A^*$ denote the union of $A'$ together with all the vertices contained in one of the endsequences $A'_i$ or $B'_i$ with $i=1,\dots,|\mathcal{P}|+1$.
Recall that $|A|=|\A|=n/(3(s+1))$ and so $|A^*|\ge |A'| \ge |A|-s^2 \ge n/(4(s+1))$. Moreover, $s\ge \ell$. Thus%
    \COMMENT{$\frac{n}{s^2}\le \frac{\eps_*n}{720k(s+1)}$ holds if $1440k/\eps_*\le s$. But $1440k/\eps_*\le 1/\eps^2_*$ by
(\ref{eq:eps*}) and $s\ge \ell\ge 1/\eps_*^2$ by (\ref{eq:ell}).}
$$
|\mathcal{P}|+1=\frac{|S''|}{s^2}+1\le \frac{n}{s^2}\stackrel{(\ref{eq:eps*}),(\ref{eq:ell})}{\le}
\frac{\eps_* n}{720k(s+1)}  \le \frac{\eps_* |A^*|}{180k}.
$$
Moreover, $p^*_4\ge n^{-1/k+\eps_*/2}\ge |A^*|^{-1/k+\eps_*/3}$. 
So
we may apply Lemma~\ref{lem:linking} (with $\eps_*/3$ playing the role of $\eps$) to see that a.a.s.~the random subgraph of $G_{n,p^*_4}$ induced by $A^*$
contains pairwise disjoint $(A'_i,B'_i)$-linkages~$R'_i$ with $|R'_i|\le \lceil 90k/\eps_* \rceil$ for all $i=1,\dots,|\mathcal{P}|+1$. So we may assume that such linkages exist. 
Thus the union of $C^*$ of all these linkages $R'_i$, of all the copies of $P^k_{s^2}$ in $\mathcal{P}$ and of $Q_{\A}$ forms the $k$th power
of a cycle which covers all vertices apart from some vertices in $A'$. 

Let $A''\subseteq A' \subseteq  A$ denote the set of all uncovered vertices.
For each $v\in A''$, let $A_v\in \A$ denote the unique absorber containing~$v$. Let $P_v$ denote the spine of $A_v$ and let $P'_v$ denote its
augmented path. Note that $C^*$ contains the $k$th power $P^k_v$ of $P_v$ as
a subgraph. But the $k$th power $(P'_v)^k$ of $P'_v$ has the same endsequences as $P^k_v$.
Thus the graph obtained from $C^*$ by replacing $P^k_v$ with $(P'_v)^k$ for each $v\in A''$ is the $k$th power of a Hamilton cycle
in $G_{n,p^*}$. (Note that our construction implies that a.a.s.~$G_{n,p*}$ contains $C^*$ as well as $(P'_v)^k$ for every $v\in A$.)


\section{Finding a factor of $k$th powers of paths: Proof of Lemma~\ref{lem:pathfactor}}\label{sec:pathfactor}

The \emph{$1$-density} of a graph $H$ on at least two vertices 
is defined to be 
$$ d_1 (H) := {e_H \over v_H -1},$$
where $e_H$ and $v_H$ denote the number of edges and the number of vertices of $H$. Let 
$$d_1^{\rm max} (H):=\max_{H'\subseteq H, \ v_{H'}\ge 2} d_1(H').$$
Lemma~\ref{lem:pathfactor} will be an easy consequence of the following deep result of Johansson, 
Kahn and Vu~\cite{JKV08}, which was already mentioned in the introduction.
\begin{theorem}[Theorem~2.2~\cite{JKV08}] \label{thm:JKV}
Fix $\varepsilon >0$ and a graph $H$. Suppose that
$p(n) \ge n^{-1/d_1^{\rm max}(H) + \varepsilon}$. Then a.a.s. $G_{n,p}$ contains an $H$-factor. 
\end{theorem}

Thus in order to prove Lemma~\ref{lem:pathfactor}, it suffices prove the following proposition.
\begin{proposition}\label{lem:Balanced} Let $k, r\ge 2$ be integers.
Then $d_1^{\rm max} (P^k_r)\le k$.
\end{proposition}
\proof
Consider any $H\subseteq P^k_r$ on $v_H\ge 2$ vertices.
Thus there is an ordering $x_1,\dots,x_{v_H}$ of the vertices of
$H$ such that for all $i=2,\dots,v_H$ every $x_i$ has at most $k$ neighbours amongst $x_1,\dots,x_{i-1}$.
Since $d_1(H[\{x_1,x_2\}])\le 2\le k$, it follows that $d_1(H)\le k$.
\endproof

It seems likely that our use of Theorem~\ref{thm:JKV} is not essential and our arguments can be extended to avoid its use. 
Indeed, first note that we only use Theorem~\ref{thm:JKV} to prove Lemmas~\ref{lem:Absorber_Factor} and~\ref{lem:pathfactor}.
As mentioned earlier, instead of Lemma~\ref{lem:Absorber_Factor}, we only need an assertion which guarantees a linear
number of disjoint absorbers. Such an assertion can be deduced from Lemma~\ref{lem:Absorber_1_Density} and a `non-partite' version of Lemma~\ref{DGbound}.
Moreover, instead of the factor covering all vertices of $S''$ guaranteed by Lemma~\ref{lem:pathfactor}, one can use this version repeatedly to cover almost all the vertices
of $S''$. The strategy would then be to use  Lemmas~\ref{lem:linking} and~\ref{DGbound} to cover the remaining vertices of $S''$ by powers of paths which are also allowed to use some vertices in $A$.
But relying on Theorem~\ref{thm:JKV} makes these steps unnecessary.%
\COMMENT{Note that if the $1$-density of a graph $H$ is close to $k$, then our bound on $p$ implies that $\Phi_H^v$ is linear.
So we can apply the original result of Kreuter to get linearly many disjoint absorbers. Similarly, later on we
can get a linear number of copies of $P^k_r$, where $r=s^2$. In fact we can repeatedly apply this result with independent copies of $G_{n,p}$
to cover all but $\eps' n$ vertices with $P^k_r$ (as in the current proof, we avoid the absorbtion vertices etc in this step).
Let $T$ be the set of these $\eps' n$ vertices. 
Using the results of the last section, find $\eps' n$ $k$-cliques each with one vertex in $T$ (so $T$ plays the role of $V_0$,
note this works as the $1$-density of a $k$-clique is less than  $k$). Now label each $k$-clique as an $k$-sequence $A_i$ and
a $k$-sequence $B_i$ and link up all these $A_i$ and $B_i$ via Lemma~\ref{lem:linking} using absorption vertices. Alternatively one can use short
paths with one endvertex in $T$ rather than cliques. In both approaches, one needs that $\eps 'n$ is much smaller than the number
of absorbtion vertices, which is ok}


\section{Finding a factor of absorbers: Proof of Lemma~\ref{lem:Absorber_Factor}}\label{sec:Absorber_Factor}

The aim of this section is to show that there are integers $j\ge 3$ and $\ell\ge 2k$ such that the 1-density of any subgraph of $A_{j,\ell,k}$
is not much larger than $k$ (see Lemma~\ref{lem:Absorber_1_Density} below).
Together with Theorem~\ref{thm:JKV} this immediately implies Lemma~\ref{lem:Absorber_Factor}.

\begin{lemma} \label{lem:Absorber_1_Density} 
For every $k\geq 2$ and every $\delta > 0$, there exist integers $j\geq 3$ and $\ell_0 \geq 2k$ such that
whenever $\ell\ge \ell_0$ every subgraph $H$ of $A_{j,\ell,k}$ satisfies $d_1(H) \leq k + \delta$.
\end{lemma}
\begin{proof}
Choose $j\ge k/\delta+3$ and $\ell_0\ge 2jk^4/\delta$.
Pick $\ell\ge \ell_0$ and let $P$ and $P'$ be the spine and the augmented path of $A_{j,\ell,k}$. 
So $A_{j,\ell,k}=P^k\cup (P')^k$. Consider any subgraph $H$ of $A_{j,\ell,k}$ on $v_H\ge 2$ vertices. Let $H^*:=H\cap P^k$.
We will distinguish the following two cases.
Roughly speaking, in the first case $H^*$ `spans' a large interval of $P^k$, in which case we can easily deduce that $d_1(H)$ is at most~$k+\delta$.

\medskip

\noindent\textbf{Case~1.} \emph{There is a component $C$ of $H^*$ satisfying one of the following properties:
\begin{itemize}
\item $V(C)\cap ({\rm Acc}(a_{i,1})\cup {\rm Acc}(a_{i,2}))\neq \emptyset$ and $V(C)\cap ({\rm Acc}(b_{i,1})\cup {\rm Acc}(b_{i,2}))\neq \emptyset$
for some $i\le j$.
\item $V(C)\cap ({\rm Acc}(b_{i,1})\cup {\rm Acc}(b_{i,2}))\neq \emptyset$ and $V(C)\cap ({\rm Acc}(a_{i+1,1})\cup {\rm Acc}(a_{i+1,2}))\neq \emptyset$
for some $i< j$.
\end{itemize}}

\smallskip

\noindent
We assume that the first property holds. The argument for the second property is similar. Note that the distance between $a_{i,2}$ and $b_{i,1}$ on $P$ is $\ell+1$
and so the distance between ${\rm Acc}(a_{i,1})\cup {\rm Acc}(a_{i,2})$ and ${\rm Acc}(b_{i,1})\cup {\rm Acc}(b_{i,2})$ on $P$ is
$\ell+1-2(k-1)$. Thus $|C|\ge (\ell-2k)/k=\ell/k-2$. Moreover, Proposition~\ref{lem:Balanced} implies that $d_1(H^*)\le k$. Thus
\begin{eqnarray*}
d_1(H)& = & \frac{e_H}{v_H-1}= \frac{e_{H^*}}{v_H-1}+ \frac{e_{H\setminus E(H^*)}}{v_H-1}
\le \frac{e_{H^*}}{v_H-1}+\frac{e(A_{j,\ell,k})-e(P^k)}{v_H-1}\\
& \stackrel{(\ref{eq:edgesabs})}{\le} & k+\frac{2j\binom{k+1}{2}+k}{\ell/k-3}
\le k+\frac{2jk^4}{\ell}\le k+\delta,
\end{eqnarray*}
as required.%
   \COMMENT{To see the 2nd inequality on the 2nd line, note that $\frac{2j\binom{k+1}{2}+k}{\ell/k-3}\le \frac{j(k+1)k+1}{\ell/2k}
\le \frac{\frac{4}{3}j(k+1)k}{\ell/2k} \le \frac{jk^3}{\ell/2k}$.}                     

\medskip

\noindent\textbf{Case~2.} \emph{There is no component of $H^*$ as in Case~1.}

\smallskip

\noindent
Let $H'$ be the spanning subgraph of $H$ whose edge set is $E(H)\setminus E(H^*)$. So every edge of $H'$ lies in $E((P')^k)\setminus E(P^k)$.
Our first aim is to choose a suitable orientation of the edges of~$H$. If $xy\in E(H^*)$ we orient $xy$ towards $y$ if and only if $y$ succeeds $x$ on~$P$.
Recall from~(A3) in Section~\ref{sec:absorber} that the subgraph $D$ of $A_{j,\ell,k}$
consisting of all junction edges, of the path $a_{1,1}vb_{1,2}$ and of all the edges $a_{i,1}a_{i,2}$, $b_{i,1}b_{i,2}$
(for all $i=1,\dots,j$) is a cycle. In order to orient the edges in $E(H')=E(H)\setminus E(H^*)$, we will use an orientation
of this cycle $D$, which we will now choose. Orient $a_{1,1}v$ towards $v$ and $vb_{1,2}$ towards $b_{1,2}$.
Since $D$ contains the path $a_{1,1}vb_{1,2}$ we can orient all edges of $D$ to obtain a directed cycle. We now use this orientation of $D$
in order to orient the edges in $E(H')$ as follows. Recall from Section~\ref{sec:absorber} that every edge in $E(A_{j,\ell,k})\setminus E(P^k)\supseteq E(H')$
is either associated with a unique junction edge or with the absorbtion vertex~$v$ of $A_{j,\ell,k}$.
For every edge $xy\in E(H')$ which is associated with some junction edge $ab$,
orient $xy$ towards $y$ if and only if $x\in {\rm Acc}(a)$ and $y\in {\rm Acc}(b)$, where $ab$ is oriented towards $b$ (in the orientation of~$D$).
Finally, for every edge $xy\in E(H')$ which is associated with $v$, orient $xy$ towards $y$ if and
only if $x\in {\rm Acc}(a_{1,1})\cup\{v\}$ and $y\in {\rm Acc}(b_{1,2})\cup\{v\}$.

Note that for every $i=2,\dots,j$, one of the junctions $a_{i,1},a_{i,2}$ sends out a junction edge while the other junction
receives a junction edge (in the orientation of~$D$). Let $a(+,i)$ denote the former junction and let $a(-,i)$ denote the latter one.
Similarly, for every $i=2,\dots,j$ one of the junctions $b_{i,1},b_{i,2}$ sends out a junction edge while the other junction
receives a junction edge. Let $b(+,i)$ denote the former junction and let $b(-,i)$ denote the latter one.
Let $a(+,1):=a_{1,1}$, $a(-,1):=a_{1,2}$, $b(+,1):=b_{1,1}$ and $b(-,1):=b_{1,2}$.
Then the following property holds for all $i=1,\dots,j$:
\textno No vertex in ${\rm Acc}(a(-,i))$ sends out an
edge in $H'$ while no vertex in ${\rm Acc}(a(+,i))$ receives an edge in $H'$. 
Similarly, no vertex in ${\rm Acc}(b(-,i))$ sends out an
edge in $H'$ while no vertex in ${\rm Acc}(b(+,i))$ receives an edge in~$H'$. & (*)

For each $i=1,\dots,j$, let $C(i,a)$ denote the union of all components of $H^*$ which intersect ${\rm Acc}(a_{i,1})\cup {\rm Acc}(a_{i,2})$
and let $C(i,b)$ denote the union of all components of $H^*$ which intersect ${\rm Acc}(b_{i,1})\cup {\rm Acc}(b_{i,2})$.
(Some of the $C(i,a)$ and $C(i,b)$ might be empty.)
Let $C^*$ denote the union of all components of $H^*$ which do not intersect any of ${\rm Acc}(a_{i,i'})$ or ${\rm Acc}(b_{i,i'})$
for $i'=1,2$ and $i=1,\dots,j$. 
Our assumption of Case~2 implies that the vertex sets of graphs $C(1,a),\dots,C(j,a), C(1,b),\dots,C(j,b),C^*$ form
a partition of $V(H^*)=V(H)\setminus \{v\}$.

Consider the vertices of $C^*$ in their order on~$P$. In the graph $H^*$ each of these vertices
sends out at most $k$ edges (in our chosen orientation). However, the last vertex of $C^*$ does not send out any edges in $H^*$. 
Thus if $C^*\neq \emptyset$ then
\begin{equation}\label{eq:edgesC*}
e(C^*)\le k|C^*|-k.
\end{equation}
Note also that none of the vertices in $C^*$ are incident to any edges of $H'$, so~(\ref{eq:edgesC*}) bounds the number of all edges of $H$ incident to vertices of $C^*$.

Let $r(i,a):= \min\{|C(i,a)|,k\}$. Consider the vertices of $C(i,a)$ in their order on~$P$. In the graph $H^*$ each of these vertices
sends out at most $k$ edges (in our chosen orientation). However, the last vertex of $C(i,a)$ does not send out any edges in $H^*$. More generally, for each
$r=0,\dots,r(i,a)-1$ the vertex of $C(i,a)$ at position $|C(i,a)|-r$ sends out at most $r$ edges in $H^*$.
Thus if $C(i,a)\neq \emptyset$ then
\begin{equation}\label{eq:edgesCia}
e(C(i,a))\le k|C(i,a)|-(k+(k-1)+\dots + (k-r(i,a)+1)).
\end{equation}
Let us now count the number of edges in $H'$ sent out by vertices of $C(i,a)$. $(*)$ implies that no vertex
in $C(i,a)-{\rm Acc}(a(+,i))$ sends out an edge in $H'$. But $a(+,i)$ sends out at most $k$ edges in $H'$.
More generally, for each $r=0,\dots,r(i,a)-1$ the unique vertex $x$ in ${\rm Acc}(a(+,i))$ which has distance $r$ from $a(+,i)$
on $P$  is incident to $k-r$ edges in $E(A_{j,\ell,k})\setminus E(P^k)\supseteq E(H')$. So $x$ sends out at most $k-r$ edges in $H'$ .
(Note that some of these $r(i,a)$ vertices $x$ of $P$ might not lie in $C(i,a)$.) Thus we have the following property:
\textno Altogether the vertices in $C(i,a)$ send out at most $k+(k-1)+\dots + (k-r(i,a)+1)$ edges
lying in the graph $H'$. &(**)

Clearly, the analogues of~(\ref{eq:edgesCia}) and~$(**)$ also hold for the $C(i,b)$. Moreover, the absorbtion vertex~$v$ of $A_{j,\ell,k}$
sends out at most $k$ edges in~$H$. Let $I^*:=1$ if $C^* \neq \emptyset$ and $I^*:=0$ otherwise. Altogether the above shows that
\begin{equation}\label{eq:edgeH}
e_H\le kv_H - k I^*.
\end{equation}
We now distinguish three subcases.

\medskip

\noindent
\textbf{Case~2a.} \emph{$C^* \neq \emptyset$.} 

\smallskip

\noindent
In this case we have 
$$d_1(H)\stackrel{(\ref{eq:edgeH})}{\le} \frac{kv_H-k}{v_H-1}=k,
$$
as required.

\medskip

\noindent
\textbf{Case~2b.} \emph{$H$ contains at least one edge associated with $v$ as well as at least one edge associated with every junction edge $ab$.} 

\smallskip

\noindent In this case we have that $v_H\ge 2j$ since there are $2j-1$ junction edges.
Thus
$$d_1(H)\stackrel{(\ref{eq:edgeH})}{\le} \frac{kv_H}{v_H-1}=k+\frac{k}{v_H-1}\le k+\frac{k}{2j-1}\le k+\delta,
$$
as required.

\medskip

\noindent
\textbf{Case~2c.} \emph{$C^*=\emptyset$. Moreover, $H$ avoids all edges
associated with~$v$ or there exists a junction edge $ab$ such that $H$ avoids all edges associated with $ab$.}

\smallskip

\noindent 
We will first show that in this case at least one of the following four properties hold:
\begin{itemize}
\item[(a)] There is an $i$ with $2\le i\le j$ such that $C(i,a)\neq \emptyset$ and $H$ avoids all edges associated with the junction edge sent out by $a(+,i)$.
\item[(b)] There is an $i$ with $1\le i\le j$ such that $C(i,b)\neq \emptyset$ and $H$ avoids all edges associated with the junction edge sent out by $b(+,i)$.
\item[(c)] $C(1,a)\neq \emptyset$ and $H$ avoids all edges associated with~$v$.
\item[(d)] $C(1,b)= \emptyset$ and $v\in V(H)$.
\end{itemize}
To prove that one of (a)--(d) holds, let $D'$ be the cycle obtained from $D$ by replacing the path $a_{1,1}vb_{1,2}$
with a single edge $e_v$ from $a_{1,1}$ to $b_{1,2}$ and contracting each edge $a_{i,1}a_{i,2}$ into a new vertex $(i,a)$ as well as contracting each edge $b_{i,1}b_{i,2}$
into a new vertex $(i,b)$ (for all $i=1,\dots,j$). Thus every edge of $D'$ apart from $e_v$ corresponds to a junction edge.
Moreover, our orientation of $D$ induces one of~$D'$. So we will view $D'$ as a directed cycle.
Colour the vertex $(i,a)$ of $D'$ red if $C(i,a)\neq \emptyset$ and colour $(i,b)$ red if $C(i,b)\neq \emptyset$.
Colour the edge $e_v$ of $D'$ red if $H$ contains some edge associated with $v$. Colour each (junction) edge $e\neq e_v$ of $D'$ red
if $H$ contains some edge associated with~$e$. Since $v_H\ge 2$  and since we are assuming that $C^*=\emptyset$, it follows that at least one vertex of $D'$ is red. Moreover, our assumption that Case~2c holds  implies that not all edges of $D'$ are red. 

Let us first consider the case when $v\notin V(H)$. Then both endvertices of a red edge of $D'$ are red.
Thus $D'$ contains a red vertex $w$ such that the edge from $w$ to its successor on $D'$ is not red.
If $w=(i,a)$ for some $i>1$ then~(a) holds. If $w=(i,b)$ for some $i\ge 1$ then~(b) holds. If $w=(1,a)$ then (c) holds.
So suppose next that $v\in V(H)$. In this case we can only guarantee that both endvertices of a red edge $e\neq e_v$ of $D'$ are red.%
    \COMMENT{We might have that $H$ contains some edge from ${\rm Acc}(a_{1,1})$ to $v$ (and so $H$ does not avoid all edges
associated to $v$) but $C(1,b)=\emptyset$, in which case $e_v$ is red, $(1,a)$ is red but $(1,b)$ is not red.}
We may also assume that $(1,b)$ is red (otherwise~(d) holds). If not all edges in $D'-e_v$
are red then $D'$ contains a red vertex $w\neq (1,a)$ such that the edge from $w$ to its successor on $D'$ is not red. Similarly as before this implies
that (a) or (b) holds. So suppose that all edges in $D'-e_v$ are red. This implies that $(1,a)$ is red and $e_v$ is not red.
Thus~(c) holds. This completes the proof that one of (a)--(d) holds. 

Suppose first that (a) holds. Then the vertices in $C(i,a)$ send out no edges lying in the graph $H'$.
Together with~(\ref{eq:edgesCia}) this implies that instead of~(\ref{eq:edgeH}) we have that
$$
e_H\le kv_H-(k+(k-1)+\dots + (k-r(i,a)+1))\le kv_H-k
$$
and so $d_1(H)\le k$ as required. The arguments for~(b) and~(c) are similar. So let us now assume that~(d) holds.
Then $v$ does not send out any edges in the graph~$H$. So instead of~(\ref{eq:edgeH}) we have $e_H\le kv_H-k$
and so $d_1(H)\le k$ as required.
\end{proof}

\section{Linking up $k$th powers of paths: Proof of Lemma~\ref{lem:linking}}\label{sec:linking}

A result of Kreuter~\cite{Kreuter} determines the threshold for the existence of a linear number of disjoint copies of a 
given graph $Q$ in a random graph $G_{n,p}$.
(This threshold is roughly the same as the one in Theorem~\ref{thm:JKV}.) 
We will prove an analogue of this result for partite multigraphs (see Lemma~\ref{DGbound}).
We will then apply Lemma~\ref{DGbound} to find disjoint copies of a partite multigraph $Q$,
where each copy of $Q_i$ of $Q$ will correspond to an $(A_i,B_i)$-linkage $R_i$ (see Lemma~\ref{partial}).
This allows us to link up a positive fraction of the pairs $(A_i,B_i)$ we are required to link up.
Roughly speaking, in the proof of Lemma~\ref{lem:linking} we will apply Lemma~\ref{partial} repeatedly to eventually 
obtain disjoint linkages for all the pairs $(A_i,B_i)$ that we are required to link up.

We write $[k]:=\{1,\dots,k\}$ and $[-k]:=\{-1,\dots,-k\}$. 
Suppose that $p=p(n)$.
We define the random graph $\cG=\cG(n_0,n,t,k,p)$ as follows:
Consider the complete $(t+1)$-partite multigraph $K$ with vertex classes $V_1,\dots,V_t$ of size $n$ and one vertex class $V_0$ of size $n_0$, and
where each edge from $V_0$ to $V_i$ has multiplicity exactly $2k$ (for all $i=1,\dots,t$) and all other edges have multiplicity one.
Moreover, for all pairs of vertices $x\in V_0$ and $y\in V_1\cup\dots\cup V_t$ the $2k$ edges between $x$ and $y$ in $K$ have labels in
$[-k]\cup [k]$ and these labels are distinct for different edges between $x$ and $y$.
We obtain $\cG$ by including each edge of $K$ into $\cG$ with probability $p$, independently of all other edges.

Let $G$ be any $(t+1)$-partite multigraph with vertex classes $Y_0,\dots,Y_t$ such that $|Y_0|\le n_0$ and $|Y_i|\le n$ for all $i=1,\dots,t$,
and where each edge between $Y_0$ and $Y_i$ has multiplicity at most $2k$
(for all $i=1,\dots,t$) and all other edges have multiplicity one.%
     \COMMENT{the graphs we are interested in will have $|Y_i|=1$. But in the second moment proofs below we need to allow $|Y_i| \le 2$.}
Moreover, for all pairs of vertices $a\in Y_0$ and $b\in Y_1\cup\dots\cup Y_t$ the edges between $a$ and $b$ in $G$ have labels in $[-k]\cup [k]$
and these labels are distinct for different edges between $a$ and $b$.

We say that a (not necessarily induced) copy of $G$ in $K$ is a \emph{good copy of $G$} if  for all $i=0,\dots,t$ each vertex in $Y_i$ is mapped to a vertex in $V_i$
and if each edge of $G$ with label~$j$ between some pair $a\in Y_0$ and $b\in Y_1\cup\dots\cup Y_t$ of vertices is mapped to the edge of $K$ with label~$j$
between the images of $a$ and $b$ in~$K$. Let $X_{G}$ denote the number of good copies of $G$ in $\cG$.
Let $D_G$ denote the maximum size of a set of disjoint good copies of $G$ in $\cG$. 
Set 
$$
\Phi_G:= \min \{ \ex( X_{G'} ) \colon G' \subseteq G, e_{G'}>0 \}
$$
and
$$
\Phi^v_G:= \min \{ \ex( X_{G'} ) \colon G' \subseteq G, v_{G'}>0 \}.
$$
Note that we allow $G'$ to consist of a single vertex in the second definition. Also note that $\Phi^v_G \le \Phi_G$.

Throughout this section, when using the $O(.)$, $\Theta(.)$ and $\Omega(.)$ notation, we mean that the size~$n$ of the vertex classes $V_1,\dots,V_t$ tends
to infinity. In most cases $n_0$ will be a function of~$n$, but we sometimes also allow $n_0=1$. The number of vertices in the graph~$G$
will always be bounded.

\begin{lemma} \label{DGbound}
Suppose that $\Phi^v_G \to \infty$ as $n\to \infty$. Then there is a constant $c>0$ (depending only on $G$) such that with probability $1-O(1/\Phi^v_G)$
we have~$D_G \ge c\Phi_G^v$.
\end{lemma}
The proof of Lemma~\ref{DGbound} is essentially the same as that of Theorem 3.29 in~\cite{JLR}, which in turn is based on an argument of 
Kreuter~\cite{Kreuter}. The difference is that in Theorem 3.29 $X_G$ counts disjoint copies of $G$ in $G_{n,p}$
(rather than good copies of $G$ in $\cG$).
For completeness, we will give a sketch which only highlights the (very minor) adjustments one has to make.
The proof of Lemma~\ref{DGbound} needs the following proposition, which is proved using a standard application of Chebyshev's inequality
(see Lemma 3.5 and Remark~3.7 in~\cite{JLR} for a similar and more detailed argument). 
Note that Proposition~\ref{XGbound} does not assume any bounds on $n_0$.
In particular, we will later also apply it in the case when $n_0=1$.

\begin{proposition} \label{XGbound}\
\begin{itemize}
\item[(i)] $Var(X_G)=O \left( \ex(X_G)^2/\Phi_G \right)$.
\item[(ii)] Suppose that $\Phi_G \to \infty$ as $n\to \infty$ and that $\eps>0$ is fixed. Then with probability
$1-O(1/\Phi_G)$ we have~$X_G = (1\pm \eps)\ex(X_G)$.
\end{itemize}
\end{proposition}
In the next two proofs we will use the following notation:
Suppose that $H$ is a $(t+1)$-partite multigraph on a bounded number of vertices with vertex classes $Y_0,\dots Y_t$ such that $|Y_0|\le n_0$.
Then we define
\begin{equation} \label{Psi}
\Psi_H:=p^{e_H} n_0^{|Y_0|} n^{v_H-|Y_0|}.
\end{equation}
Note that $\ex (X_H)= \Theta (\Psi_H)$.

\removelastskip\penalty55\medskip\noindent{\bf Proof of Proposition~\ref{XGbound}. }
Given a good copy $G'$ of $G$ in $K$, let $I_{G'}$ denote the indicator function that $G'$ is contained in $\cG$.
Below, the summations are always over good copies of the relevant graphs in $K$.
With the above notation, we have
\begin{align*}
\ex(X_G^2) & =\sum_{G',G''} \ex(I_{G'}I_{G''})\le  \ex(X_G)^2 + \sum_{E(G') \cap E(G'') \neq \emptyset} \ex(I_{G'}I_{G''})  \\
& =\ex(X_G)^2 + O \left( \sum_{ H \subseteq G, e_H >0 }  \frac{\Psi_G^2}{ \Psi_H} \right)  \\
& =\ex(X_G)^2 + O \left( \sum_{ H \subseteq G, e_H >0 } \frac{\ex(X_G)^2}{\ex(X_H)} \right) \\
& =\ex(X_G)^2 + O \left( \ex(X_G)^2/\Phi_G \right).
\end{align*}
So $Var(X_G)=O( \ex(X_G)^2/\Phi_G)$. This proves (i). 
A straightforward application of Chebyshev's inequality%
    \COMMENT{Get $\mathbb{P}(|X_G-\ex(X_G)|\ge \eps \ex(X_G))\le \frac{Var(X_G)}{\eps^2 \ex(X_G)^2}=O(\frac{1}{\eps^2\Phi_G})\to 0$.}
 now completes the proof of (ii).
\endproof

\removelastskip\penalty55\medskip\noindent{\bf Proof of Lemma~\ref{DGbound}. }
The proof begins by considering an auxiliary graph $\Gamma$, where the vertices of $\Gamma$ correspond to good copies of $G$ in $\cG$
(rather than to copies in $G_{n,p}$ as in the proof of Theorem 3.29 in~\cite{JLR}), with an 
edge between two vertices of $\Gamma$ if the corresponding copies of~$G$ share at least one vertex.
So $\Gamma$ has $X_G$ vertices and $\sum_F X_F$ edges, where the sum is taken over all unions $F= G_1 \cup G_2$
of two copies of $G$ sharing at least one vertex, and where $F$ is viewed as a $(t+1)$-partite multigraph
whose $i$th vertex class $Y_i^F$ is the union of the $i$th vertex classes of $G_1$ and $G_2$ (but we include any vertex in $G_1\cap G_2$ only once).
Since $X_F=0$ if $|Y_i^F|>n_0$, we only sum over all those $F$ for which $|Y_0^F|\le n_0$.

Note that any independent set of vertices in $\Gamma$ corresponds to a
collection of pairwise disjoint good copies of $G$ in $\cG$. So one can use Tur\'an's theorem to show that 
\begin{equation}\label{eq:turan}
D_G \ge \frac{X^2_G}{X_G+2 \sum_F X_F}.
\end{equation}
Proposition~\ref{XGbound}(ii) implies that with probability $1-O(1/\Phi_G)$, we have $\ex (X_G)/2 \le X_G \le 2\ex(X_G)$.
Together with~(\ref{eq:turan}) this implies that it suffices to show that with probability $1-O(1/\Phi^v_G)$~we have
\begin{equation} \label{eq:XF}
X_F =O \left( \frac{(\ex X_G)^2}{\Phi_G^v} \right)=O\left( \frac{\Psi^2_G}{\Phi_G^v} \right),
\end{equation}
where $\Psi_G$ is as defined in~(\ref{Psi}).
To prove (\ref{eq:XF}), the first step is to observe that if $F= G_1 \cup G_2$ is as above and $H:=G_1\cap G_2$, then
\begin{equation}\label{eq:exF}
\ex(X_F)=\Theta(\Psi_F)=\Theta\left(\frac{\Psi^2_G}{\Psi_H}\right)=O\left( \frac{\Psi^2_G}{\Phi_G^v} \right).
\end{equation}
Next, note  that Proposition~\ref{XGbound}(i) implies that
\begin{equation} \label{VarXF}
Var(X_F)=O(\Psi^2_F/\Phi_F).
\end{equation} 
To bound this expression, we need the following log-supermodularity property, where $H_1$ and $H_2$ 
are arbitrary $(t+1)$-partite multigraphs. This property follows easily from the definition of $\Psi_H$
(indeed, the overlap between $H_1$ and $H_2$ contributes twice to both the left and right hand side).
$$
\Psi_{H_1 \cup H_2} \Psi_{H_1 \cap H_2}= \Psi_{H_1}\Psi_{H_2}.
$$
Now one can proceed exactly as in the proof of Theorem 3.29:
Using repeated applications of the log-supermodularity, one can show that the%
   \COMMENT{Should we include more details?}
right hand side of~(\ref{VarXF}) is $O(\Psi_G^4/(\Phi_G^v)^3)$. 
With this bound, Chebyshev's inequality now implies that%
   \COMMENT{Get $\mathbb{P}(|X_F-\ex(X_F)|\ge \frac{\Psi_G^2}{\Phi_G^v \cdot  \ex(X_F)} \ex(X_F))\le
\frac{Var(X_F)}{\ex(X_F)^2}\frac{(\Phi_G^v)^2\cdot \ex(X_F)^2}{\Psi_G^4}=O\left(\frac{1}{\Phi_G^v}\right)$}
$$
\mathbb{P} \left( X_F \ge \ex(X_F) + \frac{\Psi_G^2}{\Phi_G^v} \right)  \le Var(X_F)\cdot \frac{(\Phi^v_G)^2}{\Psi^4_G}
=O \left( 1/\Phi_G^v \right).
$$
Together with~(\ref{eq:exF}) this implies that (\ref{eq:XF}) holds with the required probability.
\endproof

We now apply the above results to find powers of paths. Let $k\ge 2$ and $s\ge 4k$.
Recall that $P^k_s$ denotes the $k$th power of a path $P_s=x_1 \dots x_s$ on $s$ vertices.
Let $Q$ be the multigraph obtained from $P^k_s$ by contracting $x_1,\dots,x_{k}, x_{s-(k-1)}, \dots ,x_s$ into a single vertex $x_0$ and
deleting any resulting loops at $x_0$ (but not removing any of the multiple edges).
So $Q$ is a multigraph on $t+1$ vertices, where $t:=s-2k$ and where $x_0$ has degree $k(k+1)$ and all other vertices have degree $2k$.
We view $Q$ as a $(t+1)$-partite multigraph with vertex class $Y_0:=\{x_0\}$ and each other vertex class $Y_1,\dots,Y_t$ also consisting of a single vertex.
Note that every edge of $Q$ corresponds to a unique edge of $P^k_s$.
We now assign each edge of $Q$ at $x_0$ a label as follows: For all $i\in [k]$ and every $x\in Y_1\cup \dots \cup Y_t$
we label an edge of $Q$ between $x_0$ and $x$ which corresponds to an edge of $P^k_s$ between $x_i$ and $x$ with~$i$.
Similarly, for all $i\in [-k]$ and every $x\in Y_1\cup \dots \cup Y_t$
we label an edge of $Q$ between $x_0$ and $x$ which corresponds to an edge of $P^k_s$ between $x_{s+1+i}$ and $x$ with~$i$.
So for each $i\in [k]$ there are $k-i+1$ edges with labels $i,\dots,k$ between $x_0$ and $x_{k+i}$.
Similarly, for each $i\in [-k]$ there are $k+1+i$ edges with labels $-k,\dots, i$ between $x_0$ and $x_{s-(k-1)+i}$.

\begin{lemma} \label{Phibound}
Let $s> 8k^2$ and%
   \COMMENT{If $s\le 8k^2$ then the bound on $p$ makes no sense.}
define $Q$ as above. Suppose that $1 \le n_0 \le n$ and 
$p=p(n) \ge n^{-1/k+8k/s}$. Then 
\begin{itemize}
\item[(i)] $\Phi_Q =\Omega( n^{8k^2/s})$; 
\item[(ii)] $\Phi_Q^v =\Omega(n_0)$.
\end{itemize}
\end{lemma}
\begin{proof}
Note that both assertions follow if we can show that any submultigraph $Q'$ of $Q$, which contains at least one edge,
satisfies $\ex(X_{Q'})=\Omega( n^{8k^2/s} n_0)$. Let $v:=v_{Q'}$ and $e:=e_{Q'}$.

First suppose that $v \ge s/(2k)$.
In this case, it suffices to note that at most one vertex of $Q'$ has degree at most $k(k+1)$ and all others vertices of $Q'$ have degree at most $2k$.
Thus $e \le kv +k^2$ with room to spare. So recalling that $n \ge n_0 \ge 1$, we have
$$
\ex(X_{Q'})=\Omega\left( p^e n_0 n^{v-1}\right) =\Omega\left( n^{(v+k)(-1+8k^2/s)+v-1}\right).
$$
But
$$
(v+k)(-1+8k^2/s)+v-1\ge 8k^2v/s-k-1 \ge 2,
$$ 
and so the required result follows in this case, with room to spare.

So we may assume that $v \le s/(2k)$.
Consider the ordering $x_0,x_{k+1},x_{k+2}, \dots, x_{s-k}$ of the vertices of $Q$.
The assumption on $v$ implies that there are $k$ consecutive vertices%
   \COMMENT{Otherwise $s\le 2k+(v+1)k\le 3k+s/2<s$, a contradiction}
$x_a,\dots,x_{a+k-1}$ with $k < a \le s-2k+1$ which are not contained in $Q'$. Write $x_{s+1}:=x_0$.
Now for each edge $x_ix_{i'}$ of $Q'$ with $i <i'$ we either have $0\le i<i' <a$ or $a+k \le i<i' \le s+1$.
In the first case, we orient $x_ix_{i'}$ towards $x_i$ and in the second case we orient $x_ix_{i'}$ towards $x_{i'}$.
Now it is easy to see that for every vertex $x_i$ of $Q'$, the outdegree of $x_i$ in this orientation of $Q'$ is at most $k$.
Moreover, the above process yields an orientation of all edges of $Q'$
and there is at least one vertex in $Q'$ which has outdegree $0$.
(If $Q'$ contains $x_0=x_{s+1}$, then this will be one such vertex. If $Q'$ is disconnected, there will be several such vertices.)
Thus $e \le k(v-1)$.
So using that $n \ge n_0$ and $v \ge 2$, we have  
$$
\ex(X_{Q'})=\Omega\left( p^e n_0 n^{v-1}\right) =\Omega\left( n_0 (p^kn)^{v-1}\right)=\Omega\left( n_0 p^k n\right)
 =\Omega\left( n_0 n^{8k^2/s}\right),
$$
as required.
\end{proof}
We can now combine Lemma~\ref{DGbound} and Lemma~\ref{Phibound}(ii) to obtain the following result.
\begin{corollary} \label{covercor}
Let $s> 8k^2$ and define $Q$ as in Lemma~\ref{Phibound}. Suppose that $n_0 \le n$, that $n_0 \to \infty$ and
that $p=p(n) \ge n^{-1/k+8k/s}$. Then there is a constant $c>0$ (depending only on $Q$) such that 
with probability $1-O(1/n_0)$, we have~$D_Q \ge cn_0$.
\end{corollary}

Our aim is now to apply Corollary~\ref{covercor} to link up given sets of vertices in $G_{n,p}$ by powers of paths.
Suppose that $A=(a_1,\dots,a_k)$ and $B=(b_1,\dots,b_k)$ are two (ordered) sequences of vertices which are disjoint from each other.
Recall that a graph $R$ is an $(A,B)$-linkage if $R$ is obtained from the $k$th power of a path whose initial endsequence is $A$ and
whose final endsequence is $B$ by deleting all edges within $A$ and within~$B$. We call $\cA:=\{ (A_1,B_1),\dots,(A_f,B_f) \}$
a \emph{set of pairwise disjoint $k$-sequence pairs} if each $A_i$ and each $B_i$ is a sequence of $k$ vertices and
all these $2f$ sequences are pairwise disjoint. A \emph{partial $\cA$-linkage of size $f'$ and parameter $s$}
consists of $\cR=\{R_1,\dots,R_{f'}\}$ where
\begin{itemize}
\item for each $i=1,\dots,f'$ there is a $j=j(i)\in [f]$ such that $R_i$ is an $(A_j,B_j)$-linkage;
\item the $R_i$ are pairwise disjoint;
\item if $j' \neq j(i)$, then $R_i$ avoids $A_{j'}\cup B_{j'}$;
\item $|R_i|=s$ for all $i=1,\dots,f'$.
\end{itemize}
If $j'\neq j(i)$ for all $i=1,\dots,f'$, we say that $(A_{j'},B_{j'})$ \emph{is unlinked by $\cR$}.

\begin{lemma} \label{partial}
For every $0<\eps<1/k$ and every $k\ge 2$ there is a constant $c>0$ such that the following holds:
Suppose that $p=p(n) \ge n^{-1/k+\eps}$ and that $\log^2 n \le f \le n/(4k)$.
Let $\cA=\{ (A_1,B_1),\dots,(A_f,B_f) \}$ be a set of $f$ pairwise disjoint $k$-sequence pairs.
Then with probability $1-O(1/\log ^2n)$, we have that~$G_{n,p}$ contains a partial $\cA$-linkage $\cR=\{R_1,\dots,R_{f'}\}$ of size
$f':=c f$ and parameter $\lceil 10k/\eps\rceil$.
\end{lemma}
\proof
Set $s:=\lceil 10k/\eps\rceil$. So each $R_i$ will consist of $s$ vertices (including those in the endsequences of~$R_i$). Note that
the number of vertices contained in some $A_i$ or $B_i$ is $2kf \le n/2$.
We will view $G_{n,p}$ as a subgraph of $K_n$.
Let $N'$ consist of $n/2$ vertices of $K_n$ which are not contained in any of the $A_i$ or $B_i$. 
Let $t:=s-2k$. Partition $N'$ into $t$ classes $V_1,\dots,V_t$ of equal size $n':=n/(2t)$.
For all $j\in [k]$, let $V'_j$ consist of the $j$th vertex in each of the $A_i$.
So $|V'_j|=f$. For all $j\in [-k]$, let $V'_j$ consist of the $(k+1+j)$th vertex in each of the $B_i$.
Again $|V'_j|=f$.
Let $K'$ be the complete $s$-partite subgraph of $K_n$ induced by the vertex classes $V_1,\dots,V_t$,
and all the $V'_j$ for $j\in [-k]\cup [k]$. Let $K$ be the $(t+1)$-partite multigraph obtained from $K'$ by
contracting all the vertices in $A_i\cup B_i$ into a single vertex $y_i$, where
any resulting loops at $y_i$ are removed (but we do not remove any multiple edges).
So the vertex classes of $K$ are $V_0:=\{y_1,\dots,y_f\}$ and $V_1,\dots,V_t$. 
Note that each edge $e$ of $K$ corresponds to a unique edge $e'$ of $K'$.
We now label the edges of $K$ as follows: For all $i\in [f]$ and all $j\in [-k]\cup [k]$ we label an edge $e$ of $K$ between $y_i$
and some vertex $x\in V_1\cup\dots\cup V_t$ with $j$ if the corresponding edge $e'$ of $K'$ joins some vertex in $V'_j$ to $x$.

Now define a random graph $\cG$ as follows:
$\cG$ is a spanning subgraph of $K$, where we include an edge $e$ of $K$ into $\cG$ if and only if
the corresponding edge $e'$ of $K'$ is included in $G_{n,p}$.
This means that each edge of $K$ is included in $\cG$ with probability $p$, independently of all other edges.
So this corresponds exactly to the setting described at the beginning of the section,
with $n'$ playing the role of $n$ and $f$ playing the role of $n_0$.
 
Let $Q$ be as defined before Lemma~\ref{Phibound}.
Then a good copy of $Q$ in $K$ containing $y_i$ corresponds to an $(A_i,B_i)$-linkage in $K_n$.
(Thus a good copy of $Q$ in $\cG$ containing $y_i$ corresponds to an $(A_i,B_i)$-linkage in $G_{n,p}$.)
Similarly, a set of $\ell$ disjoint good copies 
of $Q$ in $K$ (or in $\cG$) corresponds to a partial $\cA$-linkage of size $\ell$ and parameter~$s$
in $K_n$ (or in $G_{n,p}$).

Also note that $p(n) \ge  n^{-1/k+\eps} \ge (n')^{-1/k+4\eps/5}\ge (n')^{-1/k+8k/s}$.
So we can apply Corollary~\ref{covercor} with $n'$ and $f$ playing the roles of $n$ and $n_0$ 
to see that, with with probability $1-O(1/\log ^2n)$, $G_{n,p}$ contains a partial $\cA$-linkage of parameter~$s$ and
size $cf$, where $c$ depends only on $Q$ (and thus only on $k$ and $\eps$).
\endproof

A simple consequence of the previous arguments is that we can link up a given sequence $A$ of $k$ vertices to a given 
sequence $B$ of $k$ vertices via the $k$th power of
a sufficiently long path.

\begin{lemma}\label{lem:singlelink}
Let $0<\eps<1/k$ and $k \ge 2$.
Suppose that $p \ge n^{-1/k+\eps}$ and that $A=(a_1 \dots a_k)$ and $B=(b_1 \dots b_k)$  are pairwise disjoint sequences of vertices.
Then with probability $1-O(1/\log^3 n)$, $G_{n,p}$ contains an $(A,B)$-linkage~$R$ with $|R|=\lceil 10k/\eps\rceil$.
\end{lemma}
\proof
Let $\cA:=\{(A,B)\}$ and $s:=\lceil 10k/\eps\rceil$.
We now define $K'$, $K$, $\cG$ and $Q$ exactly as in the proof of Lemma~\ref{partial}.
In particular, for all $j\in [k]$, let $V'_j$ consist of the $j$th vertex in $A$.
For all $j\in [-k]$, let $V'_j$ consist of the $(k+1+j)$th vertex in $B$.
So $V_0$ consists of a single vertex $y$ and $n_0=1$.
Again, a good copy of $Q$ in $K$ containing $y$ corresponds to an $(A,B)$-linkage in $K_n$
(with a similar correspondence between $\cG$ and $G_{n,p}$).
Moreover, again we have $p(n) \ge  (n')^{-1/k+8k/s}$, where $n':=n/(2t)$ and $t:=s-2k$.

Now Lemma~\ref{Phibound}(i) together with Proposition~\ref{XGbound}(ii) imply that with probability 
$1-O(n^{-8k^2/s})$, we have $X_Q>0$. So the error bound is at most $O(1/\log ^3 n)$ (with room to spare), as required.
\endproof

We can now combine Lemmas~\ref{partial} and~\ref{lem:singlelink} in order to prove Lemma~\ref{lem:linking}.

\removelastskip\penalty55\medskip\noindent{\bf Proof of Lemma~\ref{lem:linking}. }
Since $p=p(n)=o(1)$, we can view $G_{n,p}$ as a union of $2\log^2 n$ independent random graphs $G_{n,p_i}$, with $p_i =p'$,
where $p'\ge (1+o(1))p/(2 \log^2 n)\ge n^{-1/k+\eps/2}$. Let $s:=\lceil 30k/\eps\rceil$ and $\cA:=\{(A_1,B_1),\dots,(A_f,B_f)\}$.
Our strategy is to first apply Lemma~\ref{partial} repeatedly to obtain partial linkages until the number of unlinked pairs
in~$\cA$ is less than $\log^2n$ (using a different $G_{n,p_i}$ each time). We will then apply Lemma~\ref{lem:singlelink}
repeatedly in order to link the remaining pairs in~$\cA$ one by one (again, using a different $G_{n,p_i}$ each time).

Let $c=c(k,\eps)$ be as in Lemma~\ref{partial} and let $\cA_0:=\cA$. Suppose that we have obtained a set $\cA_i$ consisting of $(1-c)^i f$ unlinked pairs from~$\cA$
and that we have found a partial $\cA$-linkage $\cR_i$ with parameter~$s$ which links precisely all the pairs in $\cA\setminus \cA_i$.
Let $N_i$ be obtained from $[n]$ by deleting all the vertices in linkages from $\cR_{i}$.
Thus $|N_i|=n- (|\cA|-|\cA_i|)s\ge n-fs \ge n/2$ and so
$p_i=p'\ge n^{-1/k+\eps/2}\ge |N_i|^{-1/k+\eps/3}$. Hence if $|\cA_i|=(1-c)^if>\log^2 n$, we can apply Lemma~\ref{partial} with $\eps/3$ playing the role
of $\eps$ and with the random subgraph of $G_{n,p_i}$ induced by the set~$N_i$ playing the role of $G_{n,p}$. With probability $1-O(1/\log^2 n)$
this yields a partial linkage $\cR'_i$ of size $c|\cA_i|$ and parameter~$s$. Let $\cR_{i+1}:=\cR_i\cup \cR'_i$ and
let $\cA_{i+1}$ denote the set of pairs which are still unlinked.
So $|\cA_{i+1}|=(1-c)^{i+1} f$. 

Let $i^*\ge 0$ be the smallest integer for which $(1-c)^{i^*}f\le \log^2 n$. Thus $i^*\le \log_{1/(1-c)} n$. 
Our argument shows that with probability at least $1-O(i^*/\log ^2 n)$ we can find a partial linkage $\cR_{i^*}$ of parameter~$s$
such that the set $\cA_{i^*}$ of unlinked pairs has size $|\cA_{i^*}|= (1-c)^{i^*} f$.

We will now link up the remaining pairs one by one.
For this, write $\cA_{i^*}=\{(A^*_1,B_1^*),\dots,(A^*_{f^*},B^*_{f^*}) \}$. Thus $f^* \le \log ^2n$.
Let $N_{i^*}$ be obtained from $[n]$ by deleting all the vertices in linkages from $\cR_{i^*}$.
Suppose that $1 \le j  \le f^*$ and that we have obtained 
an $(A^*_i,B^*_i)$-linkage $R_i^*$ for all $i=1,\dots, j-1$ such that all the $R_i^*$ are pairwise disjoint,
$|R_i^*|=s$, $V(R^*_i)\subseteq N_{i^*}$ and such that $R_i^*$ avoids $(A^*_{i'},B_{i'}^*)$ for all $i'\neq i$.
Let 
$$
N^*_j:=\left(N_{i^*}\setminus \left(V(R^*_1)\cup \dots\cup V(R^*_{j-1})\cup \bigcup_{i=1}^{f^*} (A^*_i\cup B^*_i)\right)\right) \cup A^*_{j} \cup B^*_{j}.
$$ 
Thus  $|N^*_j|\ge n - fs \ge n/2$ and so $p'\ge n^{-1/k+\eps/2}\ge |N^*_j|^{-1/k+\eps/3}$.
Hence we can apply Lemma~\ref{lem:singlelink} with $\eps/3$ playing the role
of $\eps$ and with the random subgraph of $G_{n,p_{i^*+j}}$ induced by the set~$N^*_j$ playing the role of $G_{n,p}$.
With probability $1-O(1/\log ^3 n)$ this yields a $(A_j^*,B_j^*)$-linkage $R_j^*$ with $|R^*_j|=s$ and $V(R^*_j)\subseteq N^*_j$.

Since $f^* \le \log ^2n$, this means that altogether, a.a.s.~we can find pairwise disjoint $(A^*_i,B^*_i)$-linkages
for all $i=1,\dots,f^*$ which only use vertices in $N_{i^*}$ and so are disjoint from the linkages in~$\cR_{i^*}$.
\endproof

\section{Deriving Theorem~\ref{thm:riordan}}\label{sec:riordan}

Given a graph $H$ on at least three vertices, we define 
$$ d_2 (H) := {e_H \over v_H -2} \ \ \text{   and   } \ \ d_2^{\rm max} (H):=\max_{H'\subseteq H, \ v_{H'}\ge 3} d_2(H').$$
The purpose of this section is to derive Theorem~\ref{thm:riordan} from the following result of Riordan~\cite{riordan}.
Actually the result in~\cite{riordan} is more general than the version below, as its formulation in~\cite{riordan}
does not require the maximum degree of the $H_n$ to be bounded.
Moreover, it is stated for $G_{n,m}$ with $m=p\binom{n}{2}$ instead of $G_{n,p}$. But Theorem~2.2(ii) of~\cite{Bollobasbook} allows us to apply it to $G_{n,p}$.

\begin{theorem}\label{thm:riordan2}
Let $(H_n)_{n=1}^\infty$ be a fixed sequence of graphs such that $n=v_{H_n}$, $e_{H_n}\ge n$ and such that the maximum degree of
the $H_n$ is bounded. Let $p=p(n)$ be such that
\begin{equation} \label{eq:riordan}
np^{d_2^{\rm max} (H)}\to \infty, \ \ \ pn^2\to \infty\ \  \text{  and  }\ \ (1-p)\sqrt{n}\to \infty.
\end{equation}
Then a.a.s.~$G_{n,p}$ contains a copy of~$H_n$.
\end{theorem} 

Thus in order to derive Theorem~\ref{thm:riordan} from this, it suffices to prove the following proposition.
Note that the third condition in~(\ref{eq:riordan}) does not hold if $p$ is very close to $1$. But since the property of containing the 
$k$th power of a Hamilton cycle is monotonically non-decreasing under the addition of edges, this case follows immediately from the fact that
in our case there is some $p$ satisfying all three conditions 
in~(\ref{eq:riordan}).
\begin{proposition}\label{prop:2dens} Suppose that $n\ge 4k$ and $k\ge 3$. Then $d_2^{\rm max} (C^k_n)\le k+\frac{(k+1)k^2}{n}$.
Moreover, $d_2^{\rm max} (C^2_n)=3$ if $n\ge 18$.
\end{proposition}
\proof
Let us first consider the case when $k\ge 3$. Consider any $H\subseteq C^k_n$ on $v_H\ge 3$ vertices.
Suppose first that $H\subseteq P^k_n$. Thus there is an ordering $x_1,\dots,x_{v_H}$ of the vertices of
$H$ such that for all $i=2,\dots,v_H$ every $x_i$ has at most $k$ neighbours amongst $x_1,\dots,x_{i-1}$.
Since $d_2(H[\{x_1,x_2,x_3\}])\le 3\le k$, it follows that $d_2(H)\le k$.
Now suppose that $H\not\subseteq P^k_n$. Then $v_H\ge n/k$ and by deleting at most $\binom{k+1}{2}$ edges
from $H$ one can obtain a subgraph $H'$ with $H'\subseteq P^k_n$. Thus
$$d_2(H)\le d_2(H')+\frac{\binom{k+1}{2}}{v_H-2}\le d_2(H')+\frac{\binom{k+1}{2}}{n/k-2}\le k+\frac{(k+1)k^2}{n}$$
since $n\ge 4k$. A similar argument shows that $d_2^{\rm max} (C^2_n)=3$ if $n\ge 18$.%
   \COMMENT{If $H\subseteq P^2_n$ then $d_2(H)\le \min\{3,2v_H/(v_H-2)\}$. If $H\not\subseteq P^2_n$, then
$v_H\ge n/2$ and by deleting at most $\binom{3}{2}=3$ edges
from $H$ one can obtain a subgraph $H'$ with $H'\subseteq P^2_n$. Thus
$d_2(H)\le d_2(H')+\frac{3}{v_H-2}\le \frac{2v_H}{v_H-2}+\frac{3}{n/2-2}\le \frac{n}{n/2-2}+\frac{3}{n/2-2}\le 3$
since $n\ge 18$.}
\endproof

\section{Acknowledgements}

We are extremely grateful to Nikolaos Fountoulakis for helpful discussions throughout the project and comments on the manuscript.
We are also indebted to the referees for pointing out an error in an earlier version.

\medskip

{\footnotesize \obeylines \parindent=0pt

Daniela K\"{u}hn \& Deryk Osthus 
School of Mathematics
University of Birmingham
Edgbaston
Birmingham
B15 2TT
UK
}
\begin{flushleft}
{\it{E-mail addresses}:
\tt{\{d.kuhn, d.osthus\}@bham.ac.uk}}
\end{flushleft}

\end{document}